\documentclass[12pt]{article}

\usepackage{amsmath,amsfonts,amssymb,amsthm}

\newcommand{\halfp}{[\frac p2]}

\newcommand{\Z}{{\mathbb Z}}
\newcommand{\C}{{\mathbb C}}

\newcommand{\N}{{\mathbb N}}
\newcommand{\U}{{\mathcal U}}
\newcommand{\zp}{\Z_+}
\newcommand{\M}{\mathbb M}
\newcommand{\e}{{\bf e}}
\def\g{\mathfrak g}
\def\h{\mathfrak h}
\newcommand{\p}{\partial}
\newcommand{\vf}{\varphi}

\newcommand{\ov}{\overline}

\newcommand{\dt}{\delta}
\newcommand{\pz}{p\Z}
\renewcommand{\L}{\mathfrak{g}_{p}}
\newcommand{\el}{\underline{\ell}}
\newcommand{\ddd}{\underline{d}}
\newcommand{\0}{\underline{0}}
\newcommand{\ur}{\underline{r}}

\newcommand{\kkk}{k_1,\dots,k_{p-1}}
\newcommand{\rrr}{r_1,\dots,r_{p-1}}
\newcommand{\catn}{{\mathcal N(k,\underline{d})}}
\newcommand{\reorder}{\stackrel{\text{r}}{>}}
\newcommand{\ljp}[1]{L_{j+p( #1-d_j)}}
\newcommand{\bff}[1]{{\bf{#1}}}
\newcommand{\weit}[1]{{\bf{w}(\bff #1)}}
\newcommand{\prodl}[1]{L^{\bff #1}}
\newcommand{\vone}[1]{v_{\bff #1}}
\newcommand{\vtwo}[2]{v_{\bff #1,\bff #2}}
\newcommand{\mm}{\M'\times\M}
\newcommand{\supp}[1]{\text{supp}(#1)}
\newcommand{\indn}{\text{Ind}_{\underline{\ell}}(N)}

\newcommand{\spanc}[1]{\mathrm{span}_\C\left\{#1\right\}}

\newtheorem{thm}{Theorem}[section]
\newtheorem{prop}[thm]{Proposition}
\newtheorem{lem}[thm]{Lemma}
\newtheorem{cor}[thm]{Corollary}
\newtheorem{rmk}[thm]{Remark}
\newtheorem{definition}[thm]{Definition}

\begin{document}

	\begin{center}
		{\Large \bf  Restricted modules for gap-$p$ Virasoro algebra
and twisted modules for certain vertex algebras}
	\end{center}

\begin{center}
		{Hongyan Guo$^{a}$
			\footnote{Partially supported by
				NSFC (No.11901224)
				and NSF of Hubei Province (No.2019CFB160)},
Chengkang Xu$^{b}$
\footnote{Corresponding author. Partially supported by
				NSFC (No.11801375)
			}
			\\
			$\mbox{}^{a}$ School of Mathematics and Statistics,
and Hubei Key Laboratory of Mathematical Sciences,
			Central China Normal University, Wuhan 430079, China\\
$\mbox{}^{b}$ School of Mathematical Sciences, Shangrao Normal University, Shangrao 334100, China
		}

	\end{center}

\begin{abstract}
This paper studies restricted modules of
gap-$p$ Virasoro algebra $\L$ and their intrinsic
connection to twisted modules
of certain vertex algebras.
We first establish an equivalence between
the category of restricted $\L$-modules of level $\el$
and the category of twisted modules
of vertex algebra $V_{\mathcal{N}_{p}}(\el,0)$,
where $\mathcal{N}_{p}$ is a new Lie algebra,
$\el:=(\ell_{0},0,\cdots,0)\in\C^{\halfp+1}$,
 $\ell_{0}\in\C$ is the action of the Virasoro center.
Then we focus on the construction and
classification of simple restricted $\L$-modules
of level $\el$.
More explicitly,
we give a uniform construction of simple restricted $\L$-modules
as induced modules.
We present several equivalent characterizations
of simple restricted $\L$-modules, as locally nilpotent
(equivalently, locally finite) modules with respect to certain positive part of $\L$.
Moreover, simple restricted $\L$-modules of level $\el$ are classified.
They are either highest weight modules or simple induced modules.
At the end, we exhibit several concrete
examples of simple restricted $\L$-modules of level $\el$
(including Whittaker modules).

	\end{abstract}

\section{Introduction}
	\def\theequation{1.\arabic{equation}}
	\setcounter{equation}{0}

{\em Gap-$p$ Virasoro algebra} $\L$ ($p\in\Z_{\geq 2}$)
is the universal central extension of the Lie algebra of some
derivations over the Laurent polynomial ring
$\C[t,t^{-1}]$ of order at most one (cf. \cite{ACKP}, \cite{Xu}):
$$\mbox{span}\{t^{m+1}\frac{d}{dt},\; t^{r}\ | \ m\in p\Z,\;r\in\Z\backslash p\Z\}.$$
Following the notation used in \cite{Xu}, denote
$L_{m}=t^{m+1}\frac{d}{dt}$ for $m\in p\Z$ and $L_{r}=t^{r}$ for $r\in\Z\backslash p\Z$.
Then $\L$ is a complex Lie algebra with basis
$\{L_m,C_i\mid m\in\Z,0\leq i\leq \halfp\}$,
 subjecting to commutation relations
\begin{equation}\label{eq1.1}
	\begin{aligned}
		&[L_m,L_n]=(n-m)L_{m+n}+\dt_{m+n,0}\frac{1}{12}\left(\left(\frac{m}{p}\right)^3-\frac{m}{p}\right)C_0;\\
		&[L_m,L_r]=rL_{m+r};\ \ \ \ [L_r,L_s]=r\dt_{r+s,0}C_{\widetilde r};\ \ \ \
		[C_i,\L]=0,
	\end{aligned}
\end{equation}
where $m,n\in \pz, r,s\notin \pz, 0\leq i\leq\halfp, \widetilde{r}=\min\{\overline r,p-\overline r\},\overline{r}$ is the remainder of $r$ on division by $p$, and $\halfp$ is the largest integer no more than $\frac p2$.
Irreducible Harish-Chandra $\L$-modules are classified in \cite{Xu}.
They are either highest weight modules, lowest weight modules, or modules of intermediate series.

In the present paper, we study restricted modules for the gap-$p$ Virasoro algebra.
Restricted modules for Lie algebras are interesting and important.
On the one hand, from many known examples
(cf. \cite{FZhu}, \cite{G1}, \cite{G2}, \cite{GW}, \cite{L1}, \cite{LL}, etc.),
restricted modules of Lie algebras are closely related
to vertex algebras and their appropriate representations (e.g. modules, quasi modules, twisted modules, etc.).
On the other hand, several recent works show that simple restricted modules of Lie algebras
can be characterized as locally finite (or locally nilpotent) modules
over certain subalgebras.
And they can be constructed as induced modules,
their classifications are also possible
(cf. \cite{CG}, \cite{G2}, \cite{LPX}, \cite{LPXZ}, \cite{MZ2}, etc.).

To relate restricted $\L$-modules with vertex algebras and their representations,
we introduce a Lie algebra $\mathcal{N}_{p}$, which is
spanned by elements $T_{m}$, $N^{i}_{m}$ for $m\in\Z$,
$i=1,\ldots,p-1$ and central elements $K_{0}, \ldots, K_{\halfp}$, subjecting
to the Lie brackets
\begin{equation}\label{eq1.2}
	\begin{aligned}
		&[T_{m},T_{n}]=(m-n)T_{(m+n)}+\frac{1}{12}(m^{3}-m)\delta_{m+n,0}K_{0};\\
		&[T_{m},N^{i}_{n}]=-n N^{i}_{m+n};
		\ \ \ \
		[N^{i}_{m},N^{j}_{n}]=m\delta_{i+j,p}\delta_{m+n,0}K_{i},
	\end{aligned}
\end{equation}
where $m,n\in\Z$, $i,j=1,\ldots,p-1$,
$K_{i}:=K_{p-i}$ for $i>\halfp$.
Note that $\mathcal{N}_{2}$ is the
twisted Heisenberg-Virasoro algebra.
Structures and representations of the twisted Heisenberg-Virasoro
algebra have been studied widely (cf. \cite{ACKP}, \cite{AR1},
\cite{AR2}, \cite{B}, \cite{CG}, \cite{GW}, \cite{G1}).
For any $\el:=(\ell_{0},\ldots,\ell_{\halfp})\in\C^{\halfp+1}$,
we construct a vertex algebra $V_{\mathcal{N}_{p}}(\el,0)$
from $\mathcal{N}_{p}$.
There is a natural automorphism $\sigma_{p}$
of $\mathcal{N}_{p}$ defined by
$$\sigma_{p}(T_{m})=T_{m},\;\;\sigma_{p}(N^{i}_{m})=\xi_{p}^{m}N^{i}_{m},\;\;
\sigma_{p}(K_{0})=K_{0},\;\;\sigma_{p}(K_{i})=\xi_{p}^{2}K_{i}$$
for $m\in\Z,\; i=1,\ldots,p-1,$
where $\xi_{p}$ is a primitive $p$-th root of unity.
When $p=2$, $\sigma_{2}$-twisted $V_{\mathcal{N}_{2}}(\el,0)$-modules
have been studied in \cite{G1}.
For $p\geq 3$,
the automorphism $\sigma_{p}$ induces to an automorphism of $V_{\mathcal{N}_{p}}(\el,0)$
when $\ell_{1}=\cdots=\ell_{\halfp}=0$.
Fix $\el=(\ell_{0},0,\ldots,0)\in\C^{\halfp+1}$ for the rest
of the section.
Our first main result is that restricted $\L$-modules of level $\el$
are in one-to-one correspondence to
$\sigma_{p}$-twisted $V_{\mathcal{N}_{p}}(\el,0)$-modules.

Next we study simple restricted $\L$-modules of level
$\el$.
For any $\ddd=(d_{1},\ldots,d_{p-1})\in\Z^{p-1}$,
consider a subalgebra $\L^+(\ddd)$ of $\L$ (see (\ref{eq:gd})).
For certain $\L^+(\ddd)$-module $N$ (see Definition \ref{modN}),
we form the induced module
$$\indn=\U(\L)\otimes_{\U(\L^+(\ddd))}N.$$
We show that under certain conditions,
the induced modules $\indn$ are simple (Theorem \ref{thmsimp}).
 Then we give several equivalent characterizations
 of simple restricted $\L$-modules,
 e.g. as locally finite modules,
 or locally nilpotent modules (see Theorem \ref{thmeqch}).
 Furthermore,
 we prove that simple restricted $\L$-modules
 of level $\el$ are either highest weight modules,
 or simple induced modules $\indn$.
 At last, we exhibit some concrete examples of simple restricted $\L$-modules,
 including highest weight modules and Whittaker modules.

 The paper is organized as follows.
In Section 2, we study gap-$p$ Virasoro
algebra in the context of vertex algebras and their twisted modules.
In Section 3, we construct a large class of simple restricted $\L$-modules as induced modules.
In Section 4, we give several equivalent characterizations of simple restricted $\L$-modules.
They are locally nilpotent (equivalently, locally finite) with respect to certain subalgebras of $\L$.
Then we prove that simple restricted $\L$-modules
of level $\el$ are either highest weight modules or
 induced modules.
 Finally, in Section 5,
 we present some concrete examples of simple restricted $\L$-modules.

Throughout this paper, the symbols $\C$,
$\Z$, $\zp$, $\N$ represent the
 set of complex numbers, integers, positive integers,
 nonnegative integers, respectively.
 For a Lie algebra $L$, $\mathcal U(L)$ denotes the universal enveloping algebra of $L$.
 For narative simplicity,
we shall use the notation
that $C_{i}=C_{p-i}$ and
$K_{i}=K_{p-i}$ for $i>\halfp$ without confusion.

\section{Gap-$p$ Virasoro algebra and vertex algebras}
	\def\theequation{2.\arabic{equation}}
	\setcounter{equation}{0}

In this section, we relate the algebra $\L$ with vertex algebras and their
twisted modules.
More precisely, we construct a new Lie algebra $\mathcal{N}_{p}$
and the associated vertex algebras $V_{\mathcal{N}_{p}}(\el,0)$.
For $p\geq 3$, $\el:=(\ell_{0},0,\ldots,0)$,
we show that
restricted $\L$-modules of level $\el$
are equivalent to $\sigma_{p}$-twisted modules of the
vertex algebra $V_{\mathcal{N}_{p}}(\el,0)$,
 where $\sigma_{p}$
is an order $p$ automorphism of the vertex algebra
 $V_{\mathcal{N}_{p}}(\el,0)$.
 Note that the case of $p=2$ is studied in \cite{G1}.

In order to relate $\L$ with vertex algebras,
we rewrite the definition of $\L$ in the following way.
 For $m\in\Z$, $i=1,\ldots,p-1$,
 denote by
\begin{eqnarray}
L(m)=-\frac{1}{p}L_{pm},\;\; I^{i}(m)=-\frac{1}{p}L_{pm+i},
\;\;\ov{C}_{i}=\frac{C_{i}}{p},
\;\;\ov{C}_{0}=\frac{C_{0}}{p^{2}}.
\end{eqnarray}
Then it is straightforward to check
\begin{lem}\label{lem2.1}
The gap-$p$ Virasoro algebra $\L$ is spanned
by the basis elements $L(m)$,
$I^{i}(m)$, $\ov{C}_{j}$ for $m\in\Z$, $i=1,\ldots,p-1$,
$j=0,\ldots,\halfp$,
with the defining relations
\begin{equation}\label{eq2.1}
	\begin{aligned}
		&[L(m),L(n)]=(m-n)L(m+n)+\frac{1}{12}(m^{3}-m)\delta_{m+n,0}\ov{C}_{0};\\
		&[L(m),I^{i}(n)]=-(n+\frac{i}{p})I^{i}(m+n);\\
		&[I^{i}(m),I^{j}(n)]=(m+\frac{i}{p})\delta_{i+j,p}\delta_{m+n+1,0}\ov{C}_{i},
	\end{aligned}
\end{equation}
where $m,n\in\Z$, $i,j=1,\ldots,p-1$.
\end{lem}
We shall use notations in Lemma \ref{lem2.1} in this section.
Form generating functions
$$L(x)=\sum_{n\in\Z}L(n)x^{-n-2},\ \ \
I^{i}_{p}(x)=\sum_{n\in\Z}I^{i}(n)x^{-(n+\frac{i}{p})-1},\ \ \ \
i=1,\ldots,p-1.$$
Then the defining relations (\ref{eq2.1}) are equivalent to
\begin{eqnarray}\label{eq:formLie}
\begin{aligned}
	&[L(x_{1}), L(x_{2})]      =\displaystyle\frac{d}{dx_{2}}\Big(L(x_{2})\Big)x_{1}^{-1}
\delta\left(\frac{x_{2}}{x_{1}}\right)
	+2L(x_{2})\frac{\partial}{\partial x_{2}}x_{1}^{-1}\delta\left(\frac{x_{2}}{x_{1}}\right)
	\\
	&\hspace{3cm} +\frac{1}{12}\left(\frac{\partial}{\partial x_{2}}\right)^{3}x_{1}^{-1}\delta\left(\frac{x_{2}}{x_{1}}\right)\ov{C}_{0},                                                                    \\
	&[L(x_{1}), I^{i}_{p}(x_{2})] =\displaystyle\frac{d}{dx_{2}}\Big(I^{i}_{p}(x_{2})\Big)
x_{1}^{-1}\delta\left(\frac{x_{2}}{x_{1}}\right)
	+ I^{i}_{p}(x_{2})\frac{\partial}{\partial x_{2}}x_{1}^{-1}\delta\left(\frac{x_{2}}{x_{1}}\right),    \\
	&[I^{i}_{p}(x_{1}), I^{j}_{p}(x_{2})]
	=\delta_{i+j,p}\frac{\partial}{\partial x_{2}}\left(x_{1}^{-1}\delta\left(\frac{x_{2}}{x_{1}}\right)
\left(\frac{x_{2}}{x_{1}}\right)^{\frac{i}{p}}\right)\ov{C}_{i},
	\end{aligned}
\end{eqnarray}
for $i,j=1,\ldots,p-1.$

We now construct some $\L$-modules.
Let
	\begin{eqnarray}
(\L)_{\geq 0}=\sum_{m\geq 0}\C L(m)\oplus
\sum_{i=1}^{p-1}\sum_{n\geq 0}\C I^{i}(n)\oplus \sum_{i=0}^{\halfp}\C \ov{C}_{i},
\end{eqnarray}
which is a subalgebra of $\L.$

Equip $\C$ with a $(\L)_{\geq 0}$-module structure by letting
$L(m)$, $I^{i}(n)$ act trivially for $m\geq 1$, $n\geq 0$,
and $L(0), \ov{C}_{i}$ act as scalars $h, \ell_{i}$ for $i=0,\ldots,\halfp$, respectively.
Denote $\el=(\ell_{0},\ldots,\ell_{\halfp})\in\C^{\halfp+1}$.
Form the induced $\L$-module
\begin{eqnarray}
M_{\L}(\el, h)=\U(\L)\otimes_{\U((\L)_{\geq 0})}\C.
\end{eqnarray}
Set
${\bf 1}_{h}=1\otimes 1\subset M_{\L}(\el, h).$
Then $M_{\L}(\el, h)$ is $\C$-graded by $L(0)$-eigenvalues
	$$M_{\L}(\el, h)
=\bigoplus_{n\geq 0}M_{\L}(\el, h)_{(n+h)},$$
	where $M_{\L}(\el, h)_{(h)}
=\C {\bf1}_{h}$
and $M_{\L}(\el, h)_{(n+h)}$ is the $L(0)$-eigenspace of eigenvalue $n+h$ for $n> 0$.
More precisely, $M_{\L}(\el, h)_{(n+h)}$ has a basis consisting of
$$I^{i_{1}}(-k_{1})\cdots I^{i_{s}}(-k_{s})L(-m_{1})\cdots
      L(-m_{r}){\bf{1}}_{h}, $$
where
	$r, s\geq 0$, $i_{1},\ldots,i_{s}\in\{1,\ldots,p-1\}$,
	$m_{1}\geq\cdots\geq m_{r}\geq 1$, $k_{1}\geq\cdots\geq k_{s}\geq 1$ with $\sum\limits_{i=1}^{r}m_{i}+\sum\limits_{j=1}^{s}(k_{j}-\displaystyle\frac{i_{j}}{p})=n$.

In general, $M_{\L}(\el, h)$ is reducible as a $\L$-module.
Denote by $T_{\L}(\el, h)$ the unique maximal proper
$\L$-submodule of $M_{\L}(\el, h)$, and set
	\begin{eqnarray}
L_{\L}(\el, h)
=M_{\L}(\el, h)/T_{\L}(\el, h).
\end{eqnarray}
	Then $L_{\L}(\el, h)$ is an irreducible $\L$-module.

	\begin{definition}
\label{reslev}
		{\em A $\L$-module $W$ is called {\em restricted} if for any $w\in W$,
			$L(n)w=I^{i}(n)w=0$ for sufficiently large $n$ and $i=1,\ldots,p-1$.
			We say an $\L$-module $W$ is of {\em level} $\el$
if the central elements $\ov{C}_{i}$ act as
			scalars $\ell_{i}$ for $i=0,\ldots,\halfp.$}
	\end{definition}

	It is easy to see that $M_{\L}(\el, h)$,
$L_{\L}(\el, h)$ are restricted $\L$-modules
of level $\el$ for any $h\in\C$.

In the following, we will show that restricted $\L$-modules
are equivalent to the vertex algebra constructed from
the Lie algebra $\mathcal{N}_{p}$ and their twisted modules.

Consider the formal series
$$T(x)=\sum_{n\in\Z}T_{n}x^{-n-2},\ \ \
N^{i}(x)=\sum_{n\in\Z}N^{i}_{n}x^{-n-1} \ \; \mbox{for}\ \; i=1,\ldots,p-1.
$$
We have
\begin{eqnarray}\label{eq:formN}
\begin{aligned}
&[T(x_{1}), T(x_{2})]
=\displaystyle\frac{d}{dx_{2}}\Big(T(x_{2})\Big)x_{1}^{-1}
\delta\left(\frac{x_{2}}{x_{1}}\right)
	+2T(x_{2})\frac{\partial}{\partial x_{2}}x_{1}^{-1}\delta\left(\frac{x_{2}}{x_{1}}\right)
	\\
	&\hspace{3cm} +\frac{1}{12}\left(\frac{\partial}{\partial x_{2}}\right)^{3}x_{1}^{-1}\delta\left(\frac{x_{2}}{x_{1}}\right)
K_{0},                                                                    \\
  &[T(x_{1}), N^{i}(x_{2})]=\displaystyle\frac{d}{dx_{2}}\Big(N^{i}(x_{2})\Big)
      x_{1}^{-1}\delta\left(\frac{x_{2}}{x_{1}}\right)
	+ N^{i}(x_{2})\frac{\partial}{\partial x_{2}}x_{1}^{-1}\delta\left(\frac{x_{2}}{x_{1}}\right),
               \\
	&[N^{i}(x_{1}), N^{j}(x_{2})]=\delta_{i+j,p} \frac{\partial}{\partial x_{2}}\left(x_{1}^{-1}\delta\left(\frac{x_{2}}{x_{1}}\right)
        \right)K_{i}
\end{aligned}
\end{eqnarray}
for $i,j=1,\ldots,p-1$. Let
\begin{align*}
    &(\mathcal{N}_{p})_{(\leq 1)}=\coprod_{n\leq 1}\C T_{-n}\oplus
       \sum_{i=1}^{p-1}\coprod_{n\leq 0}\C N^{i}_{-n}
       \oplus\sum_{i=0}^{\halfp}\C K_{i},\\
    &(\mathcal{N}_{p})_{(\geq 2)}=\coprod_{n\geq 2}\C T_{-n}\oplus
          \sum_{i=1}^{p-1}\coprod_{n\geq 1}\C N^{i}_{-n}.
\end{align*}
They are graded subalgebras of $\mathcal{N}_{p}$,
and $\mathcal{N}_{p}=(\mathcal{N}_{p})_{(\leq 1)}\oplus (\mathcal{N}_{p})_{(\geq 2)}.$

Let
$\el=(\ell_{0},\ldots,\ell_{\halfp})\in\C^{\halfp+1}$.
Consider $\C$ as an $(\mathcal{N}_{p})_{(\leq 1)}$-module with $K_{i}$ acting
	as the scalar $\ell_{i}$ for $i=0,\ldots,\halfp,$ and
$\coprod_{n\leq 1}\C T_{-n}\oplus
\sum_{i=1}^{p-1}\coprod_{n\leq 0}\C N^{i}_{-n}$
	acting trivially.
	Form the induced module
 \begin{eqnarray}
	V_{\mathcal{N}_{p}}(\el,0)=\U(\mathcal{N}_{p})
	  \otimes_{\U((\mathcal{N}_{p})_{(\leq 1)})}\C.
\end{eqnarray}
Set ${\bf 1} =1\otimes 1\in V_{\mathcal{N}_{p}}(\el,0)$.
It is straightforward to show that
	$V_{\mathcal{N}_{p}}(\el,0)$ is a vertex operator algebra
with vacuum vector ${\bf 1}$
	and conformal vector $\omega=T_{-2}{\bf 1}$ (cf. Theorem 5.7.4 of \cite{LL}, etc.).
	And
$\{\omega=T_{-2}{\bf 1},N^{i}:=N^{i}_{-1}{\bf 1}\;\mbox{for}\; i=1,\ldots,p-1\}$
is a generating subset of
	$V_{\mathcal{N}_{p}}(\el,0)$.
Furthermore, $V_{\mathcal{N}_{p}}(\el,0)$ is
$\N$-graded by $T_{0}$-eigenvalues:
\begin{eqnarray*}
		V_{\mathcal{N}_{p}}(\el,0)=\coprod_{n\geq 0}V_{\mathcal{N}_{p}}(\el,0)_{(n)},
	\end{eqnarray*}
	where $V_{\mathcal{N}_{p}}(\el,0)_{(0)}=\C{\bf 1}$,
and for each $n\geq 1$, $V_{\mathcal{N}_{p}}(\el,0)_{(n)}$ has a basis
	$$N^{i_{1}}_{-k_{1}}\cdots N^{i_{s}}_{-k_{s}}T_{-m_{1}}\cdots
	   T_{-m_{r}}{\bf{1}}, $$
where
	$r, s\geq 0$, $i_{1},\ldots,i_{s}\in\{1,\ldots,p-1\}$,
	$m_{1}\geq\cdots\geq m_{r}\geq 2$, $k_{1}\geq\cdots\geq k_{s}\geq 1$
 with $\sum\limits_{i=1}^{r}m_{i}+\sum\limits_{j=1}^{s}k_{j}=n.$

\begin{rmk}{\em
For the Lie algebra $\mathcal{N}_{p}$,
we can show that restricted $\mathcal{N}_{p}$-modules of level
$\el$ (defined in the same way as Definition \ref{reslev})
are in one-to-one correspondence to modules for the vertex algebra
$V_{\mathcal{N}_{p}}(\el,0)$
(cf. Theorem 2.8 and Theorem 2.9 of \cite{GW} for the case $p=2$).
}
\end{rmk}

\begin{rmk}
		\label{Nuniversal}
		{\em	As a module for $\mathcal{N}_{p}$,
$V_{\mathcal{N}_{p}}(\el,0)$ is generated
by ${\bf 1}$ with the relations
			\begin{equation*}
			T_{m}{\bf 1}=
N^{i}_{n}{\bf 1}=0,\;\;\; K_{j}=\ell_{j}
			\end{equation*}
for $m\geq -1$,
$n\geq 0$, $i=1,\ldots,p-1$,
$j=0,\ldots,\halfp.$
			$V_{\mathcal{N}_{p}}(\el,0)$ is {\em universal}
			in the sense that for any $\mathcal{N}_{p}$-module $W$ of level
$\el$ equipped with a vector $v$ such that
			$
			T_{m}v= N^{i}_{n}v=0\; \mbox{for}\; m\geq -1, n\geq 0$,
$i=1,\ldots,p-1$,
			there exists a unique $\mathcal{N}_{p}$-module
 map $V_{\mathcal{N}_{p}}(\el,0)\longrightarrow W$ sending ${\bf 1}$ to $v$.}
	\end{rmk}

Let $\xi_{p}$ be a primitive $p$-th root of unity.
Let $\sigma_{p}:\mathcal{N}_{p}\longrightarrow \mathcal{N}_{p}$
be a linear map defined
 by
\begin{eqnarray}
\sigma_{p}(T_{m})=T_{m}, \;\sigma_{p}(N^{i}_{m})=\xi_{p}N^{i}_{m},\;
\sigma_{p}(K_{0})=K_{0}, \;\sigma_{p}(K_{i})=\xi_{p}^{2}K_{i},
\end{eqnarray}
for $m\in\Z$, $i=1,\ldots,p-1$.
Then $\sigma_{p}$ is an automorphism of the Lie algebra $\mathcal{N}_{p}$
of order $p$.
It lifts up naturally to an automorphism of the
vertex (operator) algebra $V_{\mathcal{N}_{p}}(\el,0)$
for $p=2$ with $\el$ being arbitrary, and for $p\geq 3$ with
$\el=(\ell_{0},0,\ldots,0)$.

For the case $p=2$ with $\el$ being arbitrary,
 it is proved in \cite{G1} (Theorems 4.4 and 4.5) that
restricted $\mathfrak{g}_{2}$-modules are equivalent to $\sigma_{2}$-twisted
$V_{\mathcal{N}_{2}}(\el,0)$-modules.
For the notions and related results of twisted modules
for vertex (operator) algebras, cf. \cite{L1}, etc.

In the rest of the paper,
we always assume $\el=(\ell_{0},0,\ldots,0)\in\C^{\halfp+1}$.
\begin{thm}	\label{1}
		If $W$ is a restricted $\L$-module of level $\el$,
then $W$ is a $\sigma_{p}$-twisted
		$V_{\mathcal{N}_{p}}(\el,0)$-module for
$V_{\mathcal{N}_{p}}(\el,0)$ as a vertex algebra with
\begin{align*}
	&Y_{\sigma_{p}}(T_{-2}{\bf 1}, x)=L(x)=\sum\limits_{n\in\Z}L(n)x^{-n-2},\\
	&Y_{\sigma_{p}}(N^{i}_{-1}{\bf 1}, x)=I^{i}_{p}(x)
	=\sum\limits_{n\in\Z}I^{i}(n)x^{-n-\frac{i}{p}-1},\;\;i=1,\ldots,p-1.
\end{align*}
\end{thm}
	\begin{proof}
Let $U_{W}=\{L(x), I^{i}_{p}(x), {\bf 1}_{W}\ | \ i=1,\ldots,p-1\}$,
where ${\bf 1}_{W}$ is the identity operator on $W$.
It follows from (\ref{eq:formLie})
and Lemma 2.2 of \cite{L1}
that
$L(x), I^{i}_{p}(x)$, $i=1,\ldots,p-1$,
are mutually local $\Z_{p}$-twisted vertex operators on $W$.
Hence, by Corollary 3.15 of \cite{L1},
$\langle U_{W}\rangle$ is a vertex algebra with $W$ a faithful $\sigma_{p}$-twisted module.
Using Lemma 2.11 of \cite{L1}, we see that
$Y(L(x),x_{1})$ and $Y(I^{i}_{p}(x),x_{1})$ satisfy the
relations (\ref{eq:formN})
of the Lie algebra $\mathcal{N}_{p}$.
So $\langle U_{W}\rangle$ is a $\mathcal{N}_{p}$-module
with $T_{n}, N^{i}_{n}$ acting as $L(x)_{n+1}$, $I^{i}_{p}(x)_{n}$ for
$n\in\Z$,
$i=1,\ldots,p-1$, and $K_{j}$ acting as $\ell_{j}$ for $j=0,\ldots,\halfp.$

By the universal property (Remark \ref{Nuniversal}) of
$V_{\mathcal{N}_{p}}(\el,0)$,
there exists a unique $\mathcal{N}_{p}$-module
homomorphism
		$$\psi: V_{\mathcal{N}_{p}}(\el,0)\longrightarrow
\langle U_{W}\rangle;\;\;{\bf 1}\mapsto {\bf 1}_{W}.$$
		Then we have
		$$\begin{aligned}
			\psi(\omega_{n}v)&=L(x)_{n}\psi(v)=\psi(\omega)_{n}\psi(v),\\
			\psi(N^{i}_{n}v)&=I^{i}_{p}(x)_{n}\psi(v)=\psi(N^{i})_{n}\psi(v)
		\end{aligned}$$
		for all $v\in V_{\mathcal{N}_{p}}(\underline{\ell},0)$,
$i=1,\ldots,p-1$, $n\in\Z.$
		Hence $\psi$ is a vertex algebra homomorphism.
		Therefore, $W$ is a $\sigma_{p}$-twisted
		$V_{\mathcal{N}_{p}}(\el,0)$-module with
		$Y_{\sigma_{p}}(T_{-2}{\bf 1}, x)=L(x)$,
		$Y_{\sigma_{p}}(N^{i}_{-1}{\bf 1}, x)=I^{i}_{p}(x)$,
$i=1,\ldots,p-1$.
	\end{proof}

Conversely, we have
	\begin{thm}	\label{2}
		Let $W$ be a $\sigma_{p}$-twisted
$V_{\mathcal{N}_{p}}(\el,0)$-module. Then $W$ is a restricted $\L$-module
 of level $\el$ with $L(x)=Y_{W}(T_{-2}{\bf 1},x)$,
$I^{i}_{p}(x)=Y_{W}(N^{i}_{-1}{\bf 1},x).$
	\end{thm}
	\begin{proof}
Let $j\geq 0$.
		Then
for $i,k=1,\ldots,p-1$,
\begin{eqnarray*}
\begin{aligned}
		& (T_{-2}{\bf 1})_{j}T_{-2}{\bf 1}= (j+1)T_{j-3}
		{\bf 1}+\delta_{j-3,0}\frac{(j-1)^{3}-(j-1)}{12}\ell_{0}{\bf 1},  \\
		& (T_{-2}{\bf 1})_{j}N^{i}_{-1}{\bf 1}= N^{i}_{j-2}{\bf 1},\ \
(N^{i}_{-1}{\bf 1})_{j}N^{k}_{-1}{\bf 1}
= j\delta_{i+k,p}\delta_{j-1,0}\ell_{i}{\bf 1}=0
\end{aligned}
 \end{eqnarray*}
(note that $\el=(\ell_{0},0,\ldots,0)$).
		We get (cf. equation (2.40) of \cite{L1})
\begin{align*}
		&[Y_{W}(T_{-2}{\bf 1}, x_{1}), Y_{W}(T_{-2}{\bf 1}, x_{2})]  \\
		&\  = Y_{W}(T_{-3}{\bf 1}, x_{2})x_{2}^{-1}\delta\left(\frac{x_{1}}{x_{2}}\right)
+2Y_{W}(T_{-2}{\bf 1}, x_{2})\left(\frac{\partial}{\partial x_{1}}\right)x_{2}^{-1}\delta\left(\frac{x_{1}}{x_{2}}\right)  \\
		&\ \ \ \ \ \ +\frac{1}{12}\left(\frac{\partial}{\partial x_{1}}\right)^{3}x_{2}^{-1}\delta\left(\frac{x_{1}}{x_{2}}\right)\ell_{0}{\bf 1},\\
		&[Y_{W}(T_{-2}{\bf 1}, x_{1}), Y_{W}(N^{i}_{-1}{\bf 1}, x_{2})]\\
		&\ \ \ =Y_{W}(N^{i}_{-2}{\bf 1},x_{2})
	  x_{1}^{-1}\delta\left(\frac{x_{2}}{x_{1}}\right)
+Y_{W}(N^{i}_{-1}{\bf 1}, x_{2})\left(\frac{\partial}{\partial x_{2}}\right)x_{1}^{-1}\delta\left(\frac{x_{2}}{x_{1}}\right),\\
		&[Y_{W}(N^{i}_{-1}{\bf 1}, x_{1}), Y_{W}(N^{k}_{-1}{\bf 1}, x_{2})]=0.
\end{align*}
Comparing with (\ref{eq:formLie}),
we see that $W$ is a $\L$-module of level $\el$ with
$L(x)=Y_{W}(T_{-2}{\bf 1},x)$,
$I^{i}_{p}(x)=Y_{W}(N^{i}_{-1}{\bf 1},x)$,
$i=1,\ldots,p-1$.
		That $W$ is restricted is clear.
	\end{proof}

Set
	\begin{equation*}
	\L^{(0)}=\C L(0)\oplus\C \sum_{i=0}^{\halfp}\ov{C}_{i}, \;\;\;\;
	\L^{(n)}=\C L(n)\;\;\mbox{for}\;0\neq n\in\Z,
	\end{equation*}
	\begin{equation*}
	\L^{(\frac{i}{p}+n)}=\C I^{i}(n),\;, i=1,\ldots,p-1.
	\end{equation*}
	Then $\L=\bigoplus\limits_{n\in\Z}\L^{(\frac{n}{p})}$ is a
$\displaystyle\frac{1}{p}\Z$-graded Lie algebra
with grading given by $L(0)$-eigenvalues.

By Theorem \ref{1} and Theorem \ref{2} we have the following result.
	\begin{thm}
		\label{voamod}
		The $\sigma_{p}$-twisted modules for $V_{\mathcal{N}_{p}}(\el,0)$
viewed as a vertex operator algebra
		(i.e. $\C$-graded by $L_{0}$-eigenvalues
and with the two grading restrictions)
are exactly those restricted modules
		for the Lie algebra $\L$ of
level $\el$ that are $\C$-graded by $L_{0}$-eigenvalues
and with the two grading restrictions.
		Furthermore, for any $\sigma_{p}$-twisted
$V_{\mathcal{N}_{p}}(\el,0)$-module $W$,
the $\sigma_{p}$-twisted
$V_{\mathcal{N}_{p}}(\el,0)$-submodules of $W$ are exactly
		the submodules of $W$ for $\L$,
and these submodules are in particular graded.
	\end{thm}

Hence, irreducible restricted $\L$-modules of level
$\el$ corresponds to irreducible
$\sigma_{p}$-twisted $V_{\mathcal{N}_{p}}(\el,0)$-modules.
	Recall that for any $h\in\C$, $L_{\L}(\el, h)$
	is an irreducible restricted
$\L$-module of level $\el$.
So
	it is an irreducible $\sigma_{p}$-twisted
$V_{\mathcal{N}_{p}}(\el,0)$-module.
In fact, we have
(the proof is similar to Theorem 4.7 of \cite{G1})
	\begin{thm}
		Let $\el=(\ell_{0},0,\ldots,0)\in\C^{\halfp+1}$. Then
		$\{L_{\L}(\el, h)\ | \ h\in\C\}$
is a complete list of irreducible
$\sigma_{p}$-twisted
$V_{\mathcal{N}_{p}}(\el,0)$-modules
for $V_{\mathcal{N}_{p}}(\el,0)$ viewed as a vertex operator algebra.
	\end{thm}

\section{Construction of simple restricted $\L$-modules}
\label{sec3}
\def\theequation{3.\arabic{equation}}
\setcounter{equation}{0}

In this section, we first recall some definitions and notations.
Then we give a general construction of simple restricted $\L$-modules
of level $\el$.

For any $\ddd=(d_{1},\ldots,d_{p-1})\in\Z^{p-1}$, let
\begin{eqnarray}\label{eq:gd}
\L^+(\ddd)=\spanc{L_{pi},L_{j+p(i-d_j)},C_k\mid i\in\N,1\leq j\leq p-1,0\leq k\leq \halfp}
\end{eqnarray}
and
\begin{eqnarray}
\L^-(\ddd)=\spanc{L_{-pi},L_{j+p(-i-d_j)}\mid i\in\zp,1\leq j\leq p-1}.
\end{eqnarray}
Notice that $\L^+(\ddd)$ is a subalgebra of $\L$,
while $\L^-(\ddd)$ may not.
We have $\L=\L^-(\ddd)\oplus\L^+(\ddd)$
as a vector space.
Note that $\L^+(\0)$ is exactly $(\L)_{\geq 0}$ in the previous section,
where $\0=(0,\cdots,0)\in\Z^{p-1}$.

For any $\ur=(r_0,\rrr)\in\Z^{p}$ and $0\leq j\leq p-1$, denote
$$Z(j;r_j)=\{j+pi\mid i\in\Z,\ i>r_j\}\hspace{0.5cm}\text{ and }\hspace{0.5cm}
Z(\ur)=\bigcup_{j=0}^{p-1}Z(j;r_j).$$

\begin{definition}\label{modN}
{\em
For $k\in\N$, $\ddd=(d_1,\dots,d_{p-1})\in\Z^{p-1}$.
Let {\em $\catn$ be the category of $\L^+(\ddd)$-modules}
whose objects $N$ satisfying the below two conditions:
\begin{description}
	\item[(I)] all $\ljp k, 1\leq j\leq p-1$, act injectively on $N$;
	\item[(II)] $L_iN=0$ for all $i\in Z(k,k-d_1,\dots,k-d_{p-1})$.
\end{description}
}
\end{definition}

Denote by $\N_{-}^\infty$ the set of all infinite sequences
$\bff i^{(-)}=(\dots,i_{-2},i_{-1})$ with entries
in $\N$ such that the number of nonzero entries is finite.
Similarly, let $\N_{+}^\infty$ be the set of all infinite sequences
$\bff i^{(+)}=(i_0,i_1,i_{2},\dots)$ with entries in $\N$
such that the number of nonzero entries is finite.
For any $\bff i^{(-)}=(\dots,i_{-2},i_{-1})\in\N_{-}^\infty$ and
$\bff i^{(+)}=(i_{0},i_{1},i_{2},\dots)\in\N_{+}^\infty$,
we will always write
\begin{eqnarray}
\bff i=\bff i^{(-)}\bff i^{(+)}=(\dots,i_{-2},i_{-1}, i_0,i_1,i_{2},\dots)
\end{eqnarray}
Set
\begin{eqnarray}
\N^{\infty}=\{\bff i=\bff i^{(-)}\bff i^{(+)}
\mid \ \bff i^{(+)}\in\N_{+}^\infty,\ \bff i^{(-)}\in\N_{-}^\infty\}.
\end{eqnarray}
Let $\bff 0^{(\pm)}\in\N_{\pm}^\infty$ be the sequences with all entries 0.
For $i\geq 0$, denote by
$\bff e_{i}^{(+)}\in\N_{+}^\infty$ the sequence with the
$(i+1)$-th entry from left being 1
and 0 elsewhere.
For $i<0$, denote by
 $\bff e_{i}^{(-)}\in\N_{-}^\infty$ the sequence with the $|i|$-th entry
from right being 1
and 0 elsewhere.
Let
\[
\bff e_i=
\begin{cases}
\bff 0^{(-)}\bff e_{i}^{(+)},\;\;\;\;\;\mbox{if}\;\; i\geq 0;\\
\bff e_{i}^{(-)}\bff 0^{(+)},\;\;\;\;\; \mbox{if}\;\; i< 0.
\end{cases}
\]

For any $\bff i\in \N^\infty$,
we call $\weit i=\sum_{s\in\Z}|s|\cdot i_s$ the {\em weight} of $\bff i$, and denote
\begin{eqnarray}
\prodl i=\cdots  L_{-2}^{i_{-2}}L_{-1}^{i_{-1}}L_0^{i_0}L_1^{i_1}L_2^{i_2}\cdots\in\U(\L).
\end{eqnarray}

We have the {\em lexicographic order} $>$ and
the {\em reverse lexicographic order} $\reorder$ on $\N^\infty$ defined as follows.
Let $\bff i,\bff j\in\N^\infty$.
We say $\bff j>\bff i$ if there is $s\in\Z$ such that $j_s>i_s$ and $j_k=i_k$ for all $k<s$.
We say $\bff j\reorder\bff i$ if there exists $s\in\Z$ such that $j_s>i_s$ and $j_k=i_k$ for all $k>s$.

For $0\leq j\leq p-1$, let
\begin{eqnarray}
\L(j)=\sum_{i\in\Z}\C L_{j+pi}\oplus\C C_{j}.
\end{eqnarray}
Recall that $C_j=C_{p-j}$ if $j>\frac p2$.
It is easy to see that $\L(0)$ is a Virasoro subalgebra
of $\L$, and all $\L(j), 1\leq j\leq p-1$, are ideals of $\L$.
Let
$$\M=\left\{\bff i\in \N^\infty\mid \prodl i\in\U\left(\L(0)\right)\right\}\;
\text{ and }\;
\M'=\left\{\bff i\in \N^\infty\mid \prodl i\in\U\left(\oplus_{j=1}^{p-1}\L(j)\right)\right\}.
$$

\begin{definition}
{\em
Define a {\em principal total order} $\succ$ on $\mm$ as follows.
For $(\bff i,\bff j)$, $(\bff k,\bff l)\in\mm$,
we say $(\bff i,\bff j)\succ(\bff k,\bff l)$
if one of the following three statements stands
\begin{enumerate}
	\item[(i)] $\weit j>\weit l$;
	\item[(ii)] $\weit j=\weit l$ and $\bff j\reorder\bff l$;
	\item[(iii)] $\bff j=\bff l$ and $\bff i>\bff k$.
\end{enumerate}
}
\end{definition}

Let $N\in\catn$ be a
$\L^+(\ddd)$-module with the
central element $C_0$ acts as a scalar $p^{2}\ell_0$
(so that $\ov{C}_{0}$ acts as $\ell_{0}$)
and $C_{i}$ acts trivially for $1\leq i\leq\halfp$.
Form the induced $\L$-module
\begin{eqnarray}\label{eq:ind}
\indn=\U(\L)\otimes_{\U(\L^+(\ddd))}N,
\end{eqnarray}
where $\el=(\ell_{0},0,\ldots,0)\in\C^{\halfp+1}$.
For any $v\in\indn$, by the PBW theorem,
$v$ can be uniquely written as
\begin{equation}\label{eq3.1}
	\sum_{(\bff i,\bff j)\in\mm}\prodl i\prodl j\vtwo ij,
\end{equation}
where all $\vtwo ij\in N$ and only finitely many of them are nonzero.
We emphasize that here $\prodl j$ is of the form
$\prodl j=\cdots L_{-2p}^{j_{-2p}}L_{-p}^{j_{-p}}$.
Let $\supp v=\{(\bff i,\bff j)\in\mm \mid \vtwo ij\neq0\}$.
Denote by $\deg v$
the maximal element in $\supp v$ with
respect to the principal total order $\succ$ on $\mm$,
called the {\em degree} of $v$.

\begin{lem}\label{lem3.1}
	Let $k\in\N$, $\ddd\in\Z^{p-1}$ and $N\in\catn$
on which $C_{0}$ acts as $p^{2}\ell_{0}$ and $C_i$ acts as zero for $1\leq i\leq \halfp$.
	For any $v\in$ {\em $\indn$}$\backslash N$,
let $\deg v=(\bff i,\bff j)$.\\
    {\rm (1)} If $\bff j\neq{\bf 0}$, then $s=\min\{t\mid j_{-pt}\neq 0\}>0$
		and $\deg(L_{j+p(s+k-d_j)}v)=(\bff i,\bff j-\e_{-ps})$ for any $j=1,\ldots,p-1$.\\
	{\rm (2)} If $\bff j={\bf 0}$ and $\bff i\neq{\bf 0}$,
		denote $s-pr=\max\{l\mid i_l\neq 0\}$, then $1\leq s\leq p-1, r>d_s$
		and $\deg(L_{p(r+k-d_s)}v)=(\bff i-\e_{s-pr},\bff 0)$.
\end{lem}
\begin{proof}
Since the central elements $C_1,\dots,C_{\halfp}$
act trivially on $N$, we have
$[L_m,L_n]=0$ on $\indn$ for all $m,n\notin p\Z$.
Write $v$ as in (\ref{eq3.1}).

(1)
It is obvious that $s>0$.
For any $(\bff k,\bff l)\in\supp v$,
by condition {\bf (II)}, we have $L_{j+p(s+k-d_j)}\vtwo kl=0$.
Thus
$$
L_{j+p(s+k-d_j)}\prodl k\prodl l\vtwo kl
=\prodl k [L_{j+p(s+k-d_j)},\prodl l]\vtwo kl.
$$
For $(\bff k,\bff l)$ with $\bff l=\bff j$, we have $\bff i\geq \bff k$.
It is straightforward to check that
$$
\deg{(L_{j+p(s+k-d_j)}\prodl k\prodl j\vtwo kj)}=(\bff k,\bff j-\e_{-ps})\preceq (\bff i,\bff j-\e_{-ps}),$$
and the equality holds if and only if $\bff k=\bff i$.

For $(\bff k,\bff l)$ with $\bff l\neq\bff j$,
denote $\deg{(L_{j+p(s+k-d_j)}\prodl k\prodl l\vtwo kl)}=(\bff k_1,\bff l_1)\in\mm$.
If $\weit j>\weit l$, we have
$$\mathrm{\bff w}(\bff l_1)\leq \weit l-ps<\weit j-ps
      =\mathrm{\bff w}(\bff j-\e_{-ps}),$$
hence $(\bff i,\bff j-\e_{-ps})\succ(\bff k_1,\bff l_1)$.
Suppose $\weit j=\weit l$ and $\bff j\reorder\bff l$.
Let $m=\min\{t\in\zp\mid l_{-pt}\neq0\}$.
Notice that $m\geq s$.
It is easy to see that
$$\mathrm{\bff w}(\bff l_1)=\weit l-pm\leq \weit j-ps=\mathrm{\bff w}(\bff j-\e_{-ps}),$$
and the equality holds only if $m=s$.
When $m=s$, we have $(\bff k_1,\bff l_1)=(\bff k,\bff l-\e_{-ps})$.
Since $\bff j-\e_{-ps}\reorder \bff l-\e_{-ps}$,
we have
$$(\bff k_1,\bff l_1)=(\bff k,\bff l-\e_{-ps})\prec(\bff i,\bff j-\e_{-ps}).$$
So $(\bff i,\bff j-\e_{-ps})\succ(\bff k_1,\bff l_1)$ in both cases.
Therefore, $\deg{(L_{j+p(s+k-d_j)}v)}=(\bff i,\bff j-\e_{-ps})$ as desired.

(2) Clearly we have $1\leq s\leq p-1$ and $ r>d_s$.
We may write $v=\sum_{\bff k\in\M'}\prodl k\vone k,$
where $\vone k\in N$ and only finitely many of them are nonzero.
Since $r>d_{s}$, we have $L_{p(r+k-d_{s})}N=0$ by condition {\bf (II)}.
Then the statement follows from the fact that $L_{s+p(k-d_s)}$ is injective on $N$ and the equation
\begin{eqnarray*}
	&&{}L_{p(r+k-d_s)}\prodl k\vone k =[L_{p(r+k-d_s)},\prodl k]\vone k
\\
	&&{}=(s-pr)k_{s-pr}L^{\bff k-\e_{s-pr}}L_{s+p(k-d_s)}\vone k
	+	 \text{terms of degree lower than }(\bff i-\e_{s-pr},\bff 0).
\end{eqnarray*}
\end{proof}

Now applying Lemma \ref{lem3.1} repeatedly to any nonzero $v\in\indn$,
we get a nonzero vector in $N\cap\U(\L)v$.
This proves the following theorem.
\begin{thm}\label{thmsimp}
	Let $k,\ddd, N,\el$ be as in {\em Lemma \ref{lem3.1}}.
	If $N$ is simple, then the induced $\L$-module {\em $\indn$} is simple.
\end{thm}

\section{Characterization of simple restricted $\L$-modules}
\def\theequation{4.\arabic{equation}}
\setcounter{equation}{0}

Recall that a $\L$-module $V$ is called restricted if for any $v\in V$ we have $L_kv=0$ for sufficiently large $k$.
In this section, we characterize simple restricted $\L$-modules as
$\L$-modules which are locally nilpotent (equivalently, locally finite)
 over certain positive part of $\L$.
Furthermore,
we give a classification of simple restricted $\L$-modules of level $\el=(\ell_{0},0,\ldots,0)\in\C^{\halfp+1}$.

\begin{definition}
	{\em
	Let $V$ be a module over a Lie algebra $\g$ and $x\in\g$.
	We say $x$ is {\em locally nilpotent} on $V$ if
 for any $v\in V$ there exists $n\in\N$ such that $x^nv=0$,
 and we say $x$ is {\em locally finite} on $V$
  if $\dim\sum_{i\in\N}\C x^iv<\infty$ for any $v\in V$.
  Moreover, $V$ is called a {\em locally nilpotent} $\g$-module
if for any $v\in V$ there exists $n\in\N$ such that $\g^nv=0$,
 and called a {\em locally finite} $\g$-module if
 $\dim\sum_{i\in\N}\g^iv<\infty$ for any $v\in V$.
	}
\end{definition}

For $n\in\Z_{+}$, denote
\begin{eqnarray}
\L^{(\geq n)}=\sum\limits_{i\geq n}\C L_{i},
\end{eqnarray}
which is a subalgebra of $\L$.
Now we give several
equivalent characterizations of simple restricted $\L$-modules.

\begin{thm}\label{thmeqch}
For a simple $\L$-module $V$, the following statements are equivalent:
	\begin{description}
		\item[(1)] there exists $k\in\zp$ such that all $L_i, i\geq k$, act on $V$ locally finitely.
		\item[(2)] there exists $k\in\zp$ such that all $L_i, i\geq k$, act on $V$ locally nilpotently.
		\item[(3)] there exists $n\in\zp$ such that $V$ is a locally finite $\L^{(\geq n)}$-module.
		\item[(4)] there exists $n\in\zp$ such that $V$ is a locally nilpotent $\L^{(\geq n)}$-module.
		\item[(5)] there exists $n\in\zp$ and a nonzero vector $v\in V$ such that $\L^{(\geq n)}v=0$.
		\item[(6)] $V$ is restricted.
	\end{description}
\end{thm}
\begin{proof}
	The conclusions
$(5)\Rightarrow(4)\Rightarrow(2)$,
$(3)\Rightarrow(1)$ and $(6)\Rightarrow (5)$ are clear.
In the following we prove that
$(2)\Rightarrow(5)\Rightarrow(3)$ and $(1)\Rightarrow(5)\Rightarrow(6)$.

	$(2)\Rightarrow(5)$:
Let $n\in p\N$, $n\geq k$.
	By assumption, $L_n$ acts locally nilpotently on $V$.
	Then there exists $0\neq v\in V$ such that $L_nv=0$.
For $1\leq j\leq n$, set $w_j=L_{n+j}v\in V$.
Let $m\in\zp$ be minimal such that $L_n^mw_j=0$ for all $1\leq j\leq n$.
	By induction, we can show that $L_{ni+n+j}v\in\C L_n^iw_j$ for all $i\in\N$.
	So $\L^{(\geq mn+n+1)}v=0$.

	$(5)\Rightarrow(3)$:
	By the PBW theorem and the simplicity of $V$, $V$ has a spanning
	set consisting of vectors of the form
	\begin{equation}\label{eq4.1}
		\cdots L_{n-2}^{i_{n-2}}L_{n-1}^{i_{n-1}}v,\ \ \ i_j\in\N,\  j\leq n-1.
	\end{equation}
	It suffices to show that each of the vector in (\ref{eq4.1})
 generates a finite dimensional $\L^{(\geq n)}$-submodule of $V$.
	Let $I$ be the annihilator of such a vector in $\U(\L^{(\geq n)})$.
	Let $m=\sum\limits_{j\leq n-1}|j|\cdot i_j$.
Then $0\leq m<\infty$.
 It is straightforward to check that
	$$\{L_i\mid i>m+n\}\cup\{L_{j_1}\cdots L_{j_s}\mid s\geq m+1,j_1\geq j_2\geq \dots\geq j_s>n\}\subset I,$$
	which implies that $I$ has a finite codimension in $\U(\L^{(\geq n)})$.

	$(1)\Rightarrow(5)$:
	First, for $n\in p\N$, $n\geq k$, we can find a nonzero $v\in V$
such that $L_nv=av$ for some
 $a\in\C$.
	For $j>n$, set
\begin{eqnarray}
N(j)=\sum_{l\in\N}\C L_n^lL_jv,
\end{eqnarray}
 which is finite dimensional by assumption.
	For any $l\geq 0$, we have
	\begin{eqnarray}
(j+nl-n\dt_{j,p\Z})L_{nl+n+j}v=[L_n,L_{nl+j}]v=(L_n-a)L_{nl+j}v.
\end{eqnarray}
Here the symbol $\dt_{j,p\Z}$ equals 1 if $j\in p\Z$, and 0 otherwise.
Therefore, $L_{nl+j}v\in N(j)$ implies that $L_{nl+n+j}v\in N(j)$.
	By induction on $l$ we show that  $L_{nl+j}v\in N(j)$ for all $l\in\N$.
	In particular, $\sum_{l\in\N}\C L_{nl+j}v$ is finite dimensional for all $j>n$. Hence
	\begin{eqnarray*}
\sum_{i\in\N}\C L_{n+i}v=\C L_n v+\sum_{j=n+1}^{2n}\left(\sum_{l\in\N}\C L_{nl+j}v\right)
\end{eqnarray*}
	is finite dimensional. Thus we can find an $l\in\zp$ such that
	\begin{equation}\label{eq4.2}
		\sum_{i\in\N}\C L_{n+i}v=\sum_{i=0}^l\C L_{n+i}v.
	\end{equation}
	Denote
	\begin{eqnarray*}
V'=\sum_{m_0,m_1,\dots,m_l\geq 0}\C L_n^{m_0}L_{n+1}^{m_1}\cdots L_{n+l}^{m_l}v,
\end{eqnarray*}
	which is nonzero and finite dimensional by assumption.
 Using Lie bracket of $\L$ and (\ref{eq4.2})
 we show that $V'$ is a $\L^{(\geq n)}$-module.

	Denote by $s$ the smallest nonnegative integer such that
$$(L_m+a_1L_{m+1}+\dots+a_sL_{m+s})V'=0$$
 for some
 $m\geq n, m\in p\Z, a_i\in\C$ and $a_s\neq 0$.
 If $s>0$, applying $L_m$, we get
	$$(a_1[L_m,L_{m+1}]+\dots+a_s[L_m,L_{m+s}])V'=0.$$
	This contradicts to the choice of $s$.
	Hence $s=0$ and $L_mV'=0$ for some $m\in p\zp$, $m\geq n$.
Now for any $j\geq n$, we have
	$$0=L_jL_mV'=[L_j,L_m]V'+L_mL_jV'=-(j-m\dt_{j,p\Z})L_{m+j}V'.$$
	In particular, $L_iV'=0$ for all $i>2m$.

	$(5)\Rightarrow(6)$: Fix an $n\in\zp$ and a nonzero vector $v\in V$ such that $\L^{(\geq n)}v=0$.
	For any $\bff i\in\N^\infty$, recall the weight $\weit i$,
we claim that $L_kL^{\bff i}v=0$ if $k\geq n+\weit i$.

	We prove this claim by induction on $|\bff i|:=\sum\limits_{s\in\Z}i_s$.
	If $|\bff i|=1$, write $\bff i=\e_i$ for some $i\in\Z$.
Since $k\geq|i|+n$, we get
	$$L_kL_iv=L_iL_kv+[L_k,L_i]v=0.$$
Suppose $|\bff i|>1$ and $L_tL^{\bff j}v=0$ for any $t,\bff j$
with $|\bff j|<|\bff i|$ and $t\geq n+\weit j$. Then
	\begin{eqnarray}
L_kL^{\bff i}v=\sum_{j\in\Z,\;\bff l\in\N^\infty}L^{\bff l}[L_k,L_j]L^{\bff i-\bff l-\e_j}v.
\end{eqnarray}
	For each $j$ appearing in the above equation,
we have $k+j\geq n+\bff w (\bff i-\bff l-\e_j)$
(note that $k\geq n+\weit i$).
Hence $L_kL^{\bff i}v=0$ by the inductional hypothesis.

	Since $V$ is simple, any nonzero vector $w\in V$
is a linear combination of finitely many vectors $L^{\bff i}v$
for some $\bff i\in\N^\infty$.
	By the claim, we obtain that $L_kw=0$ for sufficiently large $k$.
\end{proof}

We remark here that in the above proof of
the conclusions $(1)\Rightarrow(5)$ and $(2)\Rightarrow(5)$
the simplicity of $V$ is redundant.
Furthermore, we have the following corollary.
\begin{cor}
Let $m\in\Z_{+}$ and suppose $V$ is a $\L$-module on which $\L^{(\geq m)}$ act locally nilpotently.
 Then there exists a nonzero $v\in V$ such that $\L^{(\geq m)}v=0$.
\end{cor}
\begin{proof}
	Let $n\geq 1$ be minimal such that $\L^{(\geq n)}v=0$ for some $0\neq v\in V$.
The existence of such $n$ and $v$ is assured by Theorem \ref{thmeqch} (5),
and clearly $n\geq m$.

Suppose $n>m$. Then by the choice of $n$ we have $L_{n-1}v\neq 0$.
Since $L_{n-1}$ is locally nilpotent on $V$,
there exists $k\in\zp$ such that $w:=L_{n-1}^kv\neq0$ and $L_{n-1}w=L_{n-1}^{k+1}v=0$.
By induction on $j$ we can prove that
$L_{i}L_{n-1}^{j}v=0$ for any $i>n-1$, $j\geq0$.
In particular,
$L_iw=0$ for all $i\geq n-1$.
Thus $\L^{(\geq n-1)}w=0$, which contradicts to the choice of $n$.
This proves the corollary.
\end{proof}


To classify simple restricted $\L$-modules we shall recall some more notions.
The algebra $\L$ has a triangular decomposition
\begin{equation}\label{eq4.3}
	\L=\L^{-}\oplus\h\oplus\L^{+}
\end{equation}
where $\h=\C L_{0}\oplus\sum\limits_{j=0}^{\halfp}\C C_{j}$
 is a Cartan subalgebra of $\L$,
and $\L^{\pm}=\oplus_{\pm m>0}\C L_{m}.$
A $\L$-module $V$ is called a {\em weight module}
if $\h$ acts diagonalizably on $V$, i.e., we have the weight space decomposition
$V=\oplus_{\lambda\in\h^{*}}V_{\lambda},$ where
$$V_{\lambda}=\{v\in V\ | \ xv=\lambda(x)v\;\mbox{for all} \;x\in\h\}.$$
A weight $\L$-module $V$ is called a {\em highest weight module} if
$\L^{+}v=0$ and $V=\U(\L)v$ for some nonzero vector $v$ in $V$.

The following theorem is
a classification of simple restricted $\L$-modules with $C_1,\dots,C_{\halfp}$ acting trivially.

\begin{thm}\label{thm4.4}
	Let $V$ be a simple restricted $\L$-module with $C_1,\dots,C_{\halfp}$ acting trivially.
Then $V$ is a highest weight module,
or there exist $k\in\N$, $\ddd\in\Z^{p-1}$, $\el=(\ell_{0},0,\ldots,0)\in\C^{\halfp+1}$
and a simple $\L^+(\ddd)$-module $N\in\catn$ such that {\em $V\cong\indn$}.
\end{thm}
\begin{proof}
	For any $\underline{r}=(r_0,\rrr)\in\Z^{p}$, let
	\begin{eqnarray}
N_{r_0,\rrr}=\{w\in V\mid L_iw=0\text{ for all }i\in Z(\underline{r})\}.
\end{eqnarray}
	By Theorem \ref{thmeqch} (5),
$N_{r_0,\rrr}\neq 0$ for sufficiently large $r_0,\rrr$.
Let $k_0,\kkk$ be the smallest such that $N:=N_{k_0,\kkk}\neq0$.
Then all the elements $L_{j+pk_j}, 0\leq j\leq p-1$, act injectively on $N$.
Otherwise, suppose for example the action of $L_{1+pk_1}$ is not injective,
then $N_{k_0,k_1-1,k_2,\dots,k_{p-1}}\neq 0$ contradicts to the choice of $k_1$.

If $L_0$ has an eigenvector, say $w$ in $V$,
 then by the simplicity of $V$ we see that $V=\U(\L)w$ is a weight module.
 It follows from Lemma 1.6 in \cite{M} that $V$ is a highest weight module.
 Note that if $k_0<0$, then
$V$ admits an $L_0$-eigenvector (with eigenvalue 0),
hence $V$ is a highest weight module.
 We assume that $k:=k_0\geq 0$ in the following.

Denote by $d_i=k-k_i\in\Z$ for all $1\leq i\leq p-1$.
Let $\ddd=(d_{1},\ldots,d_{p-1})\in\Z^{p-1}$.
It is directly checked that $N$ is a $\L^+(\ddd)$-module in $\catn$
(note that $C_iV=0$ for all $i=1,\ldots,\halfp$).
There is a canonical $\L$-module epimorphism
$$\pi:\indn\longrightarrow V,\;\;1\otimes v\mapsto v\;\; \mbox{for any}\;\; v\in N.$$
It follows from Lemma \ref{lem3.1} that $\ker\pi=0$, that is, $V\cong\indn$.
Then $N$ is a simple $\L^+(\ddd)$-module by the property of induced modules.
\end{proof}

\begin{rmk}
	{\em	In the above proof, if $k=0$ and all $d_i>0,\ 1\leq i\leq p-1$,
then all $L_j,\ j>0$, act trivially on the simple $\L$-module $V$.
By a same proof as in \cite{MZ1},
one can show that $V$ is a weight module, thus a highest weight module.}
\end{rmk}


\section{Examples}
\def\theequation{5.\arabic{equation}}
\setcounter{equation}{0}

In this section, we present several examples of simple restricted $\L$-modules
of level $\el=(\ell_{0},0,\ldots,0)\in\C^{\halfp+1}$.

\subsection{Modules from the Virasoro algebra}

Note that the subalgebra $\L(0)$ of $\L$ is a Virasoro algebra.
 In \cite{MZ2} some concrete examples of simple $\L(0)$-modules
 on which a positive part of $\L(0)$ act locally nilpotently are given.
  Let $V$ be any such $\L(0)$-module.
Set $L_iV=0$ for all $i\notin p\Z$.
Then $V$ becomes a simple restricted $\L$-module.


\subsection{Simple $\L^+(\ddd)$-modules with $k=0$}

In this subsection we fix $k=0$.
Let $\ddd=(d_1,\dots,d_{p-1})\in\Z^{p-1}$ be
such that at least one of $d_i$ is not positive.
Let $S=\{i\mid 1\leq i\leq p-1,\;d_i\leq 0\}$
(then $S\neq\emptyset$)
and $\ov S=\{1,\dots,p-1\}\setminus S$.
Take $\theta=(\theta_i)_{i\in S}$
and $\eta=(\eta_j)_{j\in\ov S}$,
where $\theta_i, \eta_j$ are nonzero complex numbers.

Let $\ell_{0}\in\C$ and let $R(\theta,\eta)$
denote the $\L^+(\ddd)$-module $\U(\L^+(\ddd))/J$,
where $J$ is the left ideal of
 $\U(\L^+(\ddd))$ generated by
$$\begin{aligned}
&\ L_{i-pd_i}-\theta_i,\hspace{3mm}\ L_{r-pd_r}-\eta_r,\hspace{5mm}
\ i\in S,\hspace{3mm}\ r\in\ov S,\\
	&\ L_{pj}, \ L_{t+ps}, \ C_0-p^{2}\ell_{0},\ C_1,\dots,C_{\halfp}
\hspace{3mm}1\leq t\leq p-1, s>-d_t, \; j\in\zp.
\end{aligned}$$
Denote by $R=R(\theta,\eta)$ for simplicity.
Clearly, $R$ lies in the category $\catn$ with $k=0$.
Moreover, by the PBW theorem, $R$ is
isomorphic to $\C[L_0]$ as vector spaces.
From the equation
\begin{eqnarray}
(L_{i-pd_i}-\theta_i)L_0^n=\theta_i\sum_{j=0}^{n-1}\binom nj
   (-i+pd_i)^{n-j}L_0^j,
   \end{eqnarray}
for any $n\in\zp$, $i\in S$,
it is showed that the
$\L^+(\ddd)$-module $R$ is simple.

\subsection{Simple $\L^+(\ddd)$-modules with $k>0$}

In this subsection we fix $k\in\zp$ and
$\ddd=(d_1,\dots,d_{p-1})\in\Z^{p-1}$.

Let $\mathcal{Q}$ be the set of all
$(p+1)$-tuples $(S_0,S_1,\dots, S_{p-1},\Theta)$
where $S_i\subseteq\{1,2,\dots,k\}$, $0\leq i\leq p-1$,
and $\Theta=\{\theta_{i,j}\in\C\mid 0\leq i\leq p-1, j\in S_{i}\}$
satisfy the following conditions {\bf(Q1)} to {\bf(Q3)}:
\begin{description}
	\item[(Q1)] $k\in S_i$ and $\theta_{i,k}\neq0$ for all $0\leq i\leq p-1$.
	\item[(Q2)] for any $i\in \ov{S_0}:=\{0,1,\dots,k\}\setminus S_0$,
	we have $k-i\in S_t$ for some $1\leq t\leq p-1$.
	\item[(Q3)] for any $1\leq j\leq p-1$ and
$i\in \ov{S_{j}}:=\{0,1,\dots,k\}\setminus S_j$, we have $k-i\in S_0$.
\end{description}

For any $\ell_{0}\in\C$ and $(p+1)$-tuple
 $(S_0,S_1,\dots, S_{p-1},\Theta)\in\mathcal{Q}$,
consider the $\L^+(\ddd)$-module
$Q(\ell_{0},S_0,S_1,\dots, S_{p-1},\Theta):=\U(\L^+(\ddd))/I$,
where $I$ is the left ideal of $\U(\L^+(\ddd))$ generated by
$$\begin{aligned}
	&C_0-p^{2}\ell_{0},\ C_1,\dots, C_{\halfp},\ L_{pi}-\theta_{0,i},
\ L_{pj},\hspace{5mm} i\in S_0, j>k;\\
	&L_{t+p(l-d_t)}-\theta_{t,l},\ L_{t+p(n-d_t)},
\hspace{5mm} 1\leq t\leq p-1,\ l\in S_t,\ n>k.
\end{aligned}$$
For simplicity,
we write $Q:=Q(\ell_{0},S_0,S_1,\dots, S_{p-1},\Theta)$.
Then $Q\in\catn$.

Set
$T_1=\{L_{t+p(l-d_t)}\mid 1\leq t\leq p-1, l\in\ov{S_{t}}\}$,
 $T_2=\{L_{pi}\mid i\in\ov{S_0}\}$
 and $T=T_1\cup T_2$ (they are all finite sets).
Denote by $\C[T]$ the (non-commuting) polynomial ring with variables in $T$.
Let $\bff 1$ be the multiplicative unity in $\C[T]$.
By the PBW theorem, $Q$ is isomorphic to $\C[T]$ as vector spaces.
 For convenience we identify $Q$ with $\C[T]$.
  For a monomial $f\in\C[T]$, set
$$T(f)=\{L_i\in T\mid L_i\mbox{ divides } f\},\ \ \ \
   T_i(f)=T_i\cap T(f),\ i=1,2.$$
Let $L_m,L_n\in T(f)$. \\
{\bf Convention:} $f$ is always formulated in the way
  that $L_m$ appears in front of $L_n$
  if one of the following statements stands:
\begin{description}
	\item[(i)] $L_m\in T_1, L_n\in T_2$.
	\item[(ii)] $L_m, L_n\in T_2$ and $m>n$.
	\item[(iii)] $m=i+p(j-d_i), n=t+p(s-d_t)$ for some
$1\leq i,t\leq p-1,\ j\in\ov{S_i},\ s\in\ov{S_t}$
 such that $j>s$ or $j=s$ and $i>t$.
\end{description}
Under such a convention, we denote by $P(T)$,
$P(T_1)$
and $P(T_2)$
the set of all monomials with variables in $T$,
$T_1$, $T_2$ respectively.
Then
\begin{eqnarray}
P(T)=P(T_1)P(T_2)=\{fg\mid f\in P(T_1), g\in P(T_2)\}
\end{eqnarray}
 is a basis of $Q=\C[T]$.
Set $\deg {\bff1}=0$, and $\mbox{deg}(L_{n})=1$ for any $L_{n}\in T$.
For $f\in P(T)\setminus T$, suppose $f=L_nf_1$ for some $n\in\Z, f_1\in P(T)$,
we define the {\em degree} $\deg f$ of $f$ recursively by
\begin{eqnarray}
\deg f=1+\deg{f_1}.
\end{eqnarray}
For example, we have $\deg{L_{n_{1}}^{i_{1}}}=i_{1}$,
$\deg{L_{n_{1}}^{i_{1}}\cdots L_{n_{m}}^{i_{m}}}=i_{1}+\cdots+i_{m}$
for $i_{1},\ldots,i_{m}\in\Z_{+}$, $L_{n_{1}},\ldots,L_{n_{m}}\in T$.

\begin{prop}
	For any $\ell_{0}\in\C$
and $(S_0,S_1,\dots, S_{p-1},\Theta)\in\mathcal{Q}$,
 the $\L^+(\ddd)$-module $Q:=Q(\ell_{0},S_0,S_1,\dots, S_{p-1},\Theta)$ is simple.
\end{prop}
\begin{proof}
	Let $v$ be a nonzero vector in $Q$.
	Write $v=\sum_{g\in P(T)}a_gg$, where $a_g\in\C$
and only finite many of them are nonzero.
Denote $\supp v=\{g\in P(T)\mid a_g\neq 0\}$,
and $\deg v=\max\{\deg g\mid g\in\supp v\}$.
For any $g\in P(T)$,
let $g_i\in P(T_i),\ i=1,2$ be such that $g=g_1g_2$.
Note that $g_1,g_2$ are unique.
We claim that $\C[T_1]\cap\U(\L^+(\ddd))v\neq0$.
	If $v\notin\C[T_1]$, then the set
$X=\{j\in\ov{S_0}\mid L_{pj}\in T_2(g) \text{ for some }g\in\supp v\}$ is nonempty.
 Let $i$ be the smallest integer in $X$.
For any $g\in\supp v$ we may write $g_2=g_2'L_{pi}^n$,
where $n\geq 0$ and $L_{pi}\notin T_2(g_2')$.
Since $i\in\ov{S_0}$, by {\bf (Q2)} we see that
there exists some $1\leq t\leq p-1$ such that $k-i\in S_t$.
Notice that $[L_{t+p(k-i-d_t)},g_1g_2']=0$ on $Q$.
We have
	\begin{eqnarray}
		(L_{t+p(k-i-d_t)}-\theta_{t,k-i})v
		&= &\sum_{g\in\supp v}a_gg_1g_2'[L_{t+p(k-i-d_t)},L_{pi}^n] \nonumber\\
		&=&-\sum_{g\in\supp v}a_g\theta_{t,k}
\Big(t+p(k-i-d_t)\Big)\frac{\p g}{\p(L_{pi})},
	\end{eqnarray}
	which is obviously nonzero, and has degree less than $\deg v$.
	Repeat this progress, one gets the claim.

	Now for $v\in\C[T_1]\setminus\C\bff1$,
	write $v=\sum_{g\in P(T_1)}a_gg$.
 Let $L_{j+p(i-d_j)}\in T_1$ be the rightmost element
 in the set $\cup_{g\in\supp v}T_1(g)$ with respect to ${\bf (iii)}$.
	Now for any $g\in\supp v$ we may write as the form
	\begin{eqnarray}
g=g'\prod_{l=j}^{p-1}L_{l+p(i-d_l)}^{n_l},
\end{eqnarray}
	where $n_l\in\N$ (with $n_{j}\in\Z_{+})$ and
$g'\in P(T_1)$ is a product of
 $L_{t+p(s-d_t)}\in T_1$ with $s>i, 1\leq t\leq p-1$.
Since $i\in\ov{S_j}$, we have $k-i\in S_0$ by {\bf (Q3)}.
 Note that $[L_{p(k-i)},g']=0$ on $Q$.
  We have
	\begin{eqnarray}
\begin{aligned}
		(L_{p(k-i)}&-\theta_{0,k-i})v
=\sum_{g\in\supp v}a_gg'[L_{p(k-i)},\prod_{l=j}^{p-1}L_{l+p(i-d_l)}^{n_l}]\\
		&=\sum_{g\in\supp v}a_g\left(\sum_{l=j}^{p-1}
\theta_{l,k}\Big(l+p(i-d_l)\Big)\frac{\p g}{\p(L_{l+p(i-d_l)})}\right),
	\end{aligned}
\end{eqnarray}
	which is again nonzero and has degree less than $\deg v$.
 Repeat this computation and we see that $Q$ can be generated by any nonzero vector.
	So $Q$ is simple.
\end{proof}

The above two subsections give explicit constructions
of simple $\L^+(\ddd)$-modules $Q$ and $R$
in the category $\catn$ respectively,
from which we get simple restricted $\L$-modules
$\text{Ind}_{\el}(Q)$ and $\text{Ind}_{\el}(R)$
of level $\el=(\ell_{0},0,\ldots,0)\in\C^{\halfp+1}$.

\subsection{Whittaker modules}

In this subsection,
we study Whittaker modules over $\L$ of level
$\el$.
Irreducible Whittaker modules in certain cases
are explained to be the above constructed modules
with $0\leq k\leq 2$.

Let $\mathcal C=\spanc{C_0,C_1,\dots,C_{\halfp}}$.
Recall $\L^+=\sum\limits_{m>0}\C L_m$ from (\ref{eq4.3}).
Let $\vf:\L^+\oplus\mathcal C\longrightarrow\C$ be a Lie algebra homomorphism.
A {\em Whittaker module of type $\vf$} over $\L$ is a $\L$-module
 generated by a vector $v$ such that $x.v=\vf(x)v$ for any $x\in \L^+\oplus\mathcal C.$
Denote by $\C_\vf$ the one dimensional $(\L^+\oplus\mathcal C)$-module defined by
$$x.1=\vf(x)\hspace{5mm}\text{for any }x\in \L^+\oplus\mathcal C.$$
The induced $\L$-module
\begin{eqnarray}
W(\vf)=\U(\L)\otimes_{\U(\L^+\oplus\mathcal C)}\C_\vf
\end{eqnarray}
is called the {\em universal Whittaker module} of type $\vf$.
It is easy to see that $\vf(L_i)=0$ for $i>p$ and $i\neq 2p$.
We require $\vf(C_i)=0$ for all $1\leq i\leq \halfp$ in the following.
Then one can prove that $W(\vf)$ is irreducible
if and only if $\vf(L_i)\neq0$ for all $1\leq i\leq p-1$.

If $\vf(L_{2p})\neq 0$, then irreducible quotients of $W(\vf)$ are
the $\L$-modules constructed in Subsection 5.3 with $k=2$.
Here is an explicit example.
Suppose $\vf(L_{2p})\neq0$ and $\vf(L_i)\neq0,\ 1\leq i\leq p-1$.
Then $W(\vf)\cong\text{Ind}_{\el}(Q(\ell_0,S_0,S_1,\dots, S_{p-1},\Theta))$
with $p^{2}\ell_0=\vf(C_0),\ \el=(\ell_0,0,\dots,0)\in\C^{\halfp+1}$,
 $S_0=\{1,2\}$, $S_1=\dots=S_{p-1}=\{2\}$,
$\Theta=\{\theta_{i,j}\in\C \ | \ 0\leq i\leq p-1,j\in S_{i}\}$
 such that $\theta_{0,1}=0$, $\theta_{0,2}=\vf(L_{2p})$,
  $
 \theta_{i,2}=\vf(L_i),\ 1\leq i\leq p-1$.

If $\vf(L_{2p})=0$ and $\vf(L_{p})\neq 0$,
then irreducible quotients of $W(\vf)$
 are the $\L$-modules constructed in Subsection 5.3 with $k=1$.
When $\vf(L_{2p})=\vf(L_{p})=0$, irreducible quotients of $W(\vf)$ are
the $\L$-modules constructed in Subsection 5.2.

\end{document}